%% file: Restricted_Doppler_tomography.tex
\documentclass[12pt]{article}
\usepackage[pagebackref,colorlinks=true,urlcolor=red,linkcolor=red,citecolor=red]{hyperref}

\usepackage{amsthm, amssymb, bm}
\usepackage{soul, color}
\usepackage[]{amsmath}
\usepackage[]{amsfonts}
\usepackage[]{fancyhdr}
\usepackage[]{graphicx}
\graphicspath{{/EPSF/}{../figures/}{figures/}}

\input preamble

\newcommand{\vi}{\textbf{\textit{v}}}

\newcommand{\D}{\mathrm{d}}

\newcommand{\Rb}{\mathbb{R}}
\newcommand{\Sb}{\mathbb{S}}

\renewcommand{\O}{\Omega}

\newcommand{\g}{\boldsymbol{\gamma}}

\usepackage{amsthm}
\usepackage{amsfonts}
\newcommand{\Beq}{\begin{equation}}
\newcommand{\Eeq}{\end{equation}}
\newcommand{\beq}{\begin{equation*}}
\newcommand{\eeq}{\end{equation*}}
\newcommand{\bal}{\begin{align}}
\newcommand{\eal}{\end{align}}
\newtheorem{example}{Example}

\title{Full reconstruction of a vector field from restricted Doppler and first integral moment transforms in $\mathbb{R}^{n}$}
\author{Rohit Kumar Mishra\thanks{Department of Mathematics, University of California, Santa Cruz, CA 95064, USA. rokmishr@ucsc.edu}}

\begin{document}
\maketitle
\begin{abstract}
We show that a vector field in $\mathbb{R}^n$ can be reconstructed uniquely from the knowledge of restricted Doppler and first integral moment transforms. The line complex we consider consists of all lines passing through a fixed curve $\g \subset \mathbb{R}^n$. The question of reconstruction of a symmetric $m$-tensor field from the knowledge of the first $m+1$ integral moments was posed by Sharafutdinov in his book (see pp. 78),  ``Integral geometry of tensor fields," Inverse and Ill-posed problems series, De Grutyer. In this work, we provide an answer to Sharafutdinov's question for the case of vector fields from restricted data comprising of the first $2$ integral moment transforms.
\end{abstract}

\section{Introduction}\label{sec:introduction}
The reconstruction of a symmetric $m$-tensor field $f$ from its ray transform known along all lines or along a sub-collection of lines in $\mathbb{R}^n$ is a very important problem. It has applications in several areas such as acoustic flow imaging using time-of-flight measurements \cite{SAJohnson1977}, non-destructive evaluation \cite{Sato1985}, deflection optical tomography to determine densities in supersonic expansions and flames \cite{Faris1988}, ocean tomography \cite{Munk1979,Howe1987}, reconstruction of velocity vector fields in blood vessels \cite{Jansson1995}, boundary rigidity problem \cite{Michel, Sharafutdinov_Book, SU1} and photoelasticity \cite{Sharafutdinov_Book}. Such a reconstruction problem can also be considered in the setting of Riemannian manifolds, where integrals of a tensor field are considered over geodesics, and has important applications in geophysics \cite{Sharafutdinov_Book, SU1, SU2, Uhlmann-Vasy}.

It is well known \cite{Sharafutdinov_Book} that any symmetric $m$-tensor field can be decomposed uniquely into a potential and  a solenoidal component. The potential component of a symmetric tensor field is always contained in the kernel of the ray transform. Thus one can only hope to recover the solenoidal part of a symmetric $m$- tensor field from its ray transform. Sharafutdinov \cite{Sharafutdinov_Book} gave an explicit inversion formula to reconstruct the solenoidal part of a symmetric $m$-tensor field from the knowledge of its ray transform over all lines in $\mathbb{R}^n$. For scalar functions ($m=0$), the reconstruction problem is a classical one in mathematical tomography. Beginning with the classical work of Radon \cite{Radon_Paper} and of Cormack \cite{Cormack_Paper}, there are several inversion results; see \cite{Helgason_Book} and the references therein. In the case of incomplete data, there are several inversion results  for different values of $m$ \cite{Tuy1983,Denisjuk1994,Schuster2000,Schuster2001,Katsevich2006,Sharafutdinov2007,
Palamodov2009,Svetov2012.,Katsevich2013,Denisiuk_2018}.  Vertgeim \cite{Vertgeim2000} gave a method to reconstruct the solenoidal part from the associated ray transform over all lines intersecting a fixed curve. Denisjuk \cite{Denisjuk_Paper} showed for $n=3$ and any $m$ that the solenoidal part of a tensor field can be recovered with an explicit formula if its ray transform is known over all lines intersecting a fixed curve satisfying the so called Kirillov-Tuy condition. The condition given on the curve in Denisjuk's work is less restrictive compared to the work of Vertgeim in the sense that required number of intersection points of the curve and any hyperplane that intersects support of $f$ is lesser.

The reconstruction of the solenoidal component  in a Riemannian geometry setting has been also extensively studied; see \cite{Pestov-Sharafutdinov,Sharafutdinov_Upperbound_Curvature, Sharafutdinov_Book, SU1, SU2, SU3, Uhlmann-Vasy,Stefanov-Uhlmann-Vasy}. Furthermore, approximate inversion results such as microlocal inversion formulas have been thoroughly investigated as well in several works \cite{GS, Greenleaf-Uhlmann-Duke1989, Boman-Quinto-Duke, Boman1993,SU1, Katsevich2002, Lan2003,Ramaseshan2004, SU2, SU3, VRS} starting from the fundamental work of Guillemin-Sternberg \cite{Guillemin-Sternberg-AJM} who first studied generalized Radon transform in the framework of Fourier integral operators.

Since the full recovery of a  symmetric $m$-tensor field is not possible from its ray transform, the question that one can ask is whether it is possible to reconstruct the full tensor field given some additional data. In this connection Sharafutdinov in  \cite{Sharafutdinov_Generalized_Tensor_Fields} proved a uniqueness result showing that full recovery of a symmetric $m$-tensor field $f$ is possible from the knowledge of the so-called first $m+1$ integral moments \cite[pp.78]{Sharafutdinov_Book}. In this connection, we also point out several related works \cite{Sparr-Strahlen, Holman2010, Braun-Hauck, Svetov2014, Rohit2017,Venky_Suman_Manna,Abhishek2018,Venky_Rohit_Francois}, where the authors show full reconstruction of a vector field from the knowledge of the usual (longitudinal) and transverse ray transforms.

Our focus in the current article is to provide an inversion method to uniquely recover a  vector field if its Doppler transform and the $1^{st}$ integral moment transform (see Definition \ref{Moments_Def} below) are known  along all lines intersecting a fixed curve $\g$. To the best of our knowledge, the explicit reconstruction of the solenoidal component of the vector field from restricted data is known so far only for the 3-dimensional case. In our work, we show the reconstruction of the solenoidal component of the vector field from restricted ray transform data in any dimension. Furthermore, the full reconstruction of the vector field from restricted integral moment transforms is new and as far as we are aware of, has not been considered previously. Our proof follows closely the techniques of Denisjuk \cite{Denisjuk_Paper} to recover the solenoidal part of a vector field explicitly from the knowledge the Doppler transform known along all lines intersecting a fixed curve satisfying the so-called Kirillov-Tuy condition in $\mathbb{R}^n$. Then we solve an elliptic boundary value problem to get the solenoidal part in the bounded domain. Finally we use the restricted first integral moment transform to get the potential part. 
\par  The article is organised as follows. In the Section \ref{Preliminaries}, we  give preliminaries and the statements of our main results. In Section \ref{Proof of Solenoidal part}, we prove some important lemmas which we  use to prove our main results in the same section.
\section{Preliminaries and statements of the main results}\label{Preliminaries}
We start with introducing some notations for spaces of functions and distributions which we are going to use throughout this article.
\begin{itemize}
\item $\mathcal{D}( \mathbb{R}^n;\mathbb{R}^n)$ :  the space of covector fields in $\Rb^n$ whose components are  compactly supported smooth functions.
\item $\mathcal{D}^\prime( \mathbb{R}^n; \mathbb{R}^n)$ :  the space of covector fields in $\Rb^n$ whose components are  distributions.
\item $\mathcal{E}( \mathbb{R}^n; \mathbb{R}^n)$ :  the space of covector fields in $\Rb^n$ whose components are smooth functions.
\item $\mathcal{E}^\prime( \mathbb{R}^n; \mathbb{R}^n)$ :  the space of covector fields in $\Rb^n$ whose components are  compactly supported smooth distributions.
\end{itemize}
In the similar manner, we may define spaces of covector fields in $\Rb^n$ with components in some Hilbert space $H^s$ or in some Lebesgue space $L^p$ as well.
% Let $\mathcal{D}(T^* \mathbb{R}^n)$ denote the space of covector fields in $\Rb^n$ whose components are  compactly supported smooth functions.
\begin{definition} \label{Doppler transform}
	The  Doppler transform of a covector field $f =( f_1, \dots , f_n) \in \mathcal{D}(\mathbb{R}^n;\mathbb{R}^n)$ is the function $Df$  on the space of oriented lines parametrized by $(x, \xi) \in T \mathbb{S}^{n-1}$ and defined as the following:
\begin{align}\label{eq:definition of Doppler transform}
Df(x,\xi) = \int_{\mathbb{R}} f_{i}(x+t\xi)\xi^{i}d t.
\end{align}
Here and elsewhere with repeating indices, Einstein summation convention will be assumed. With a little abuse of notation, we denote the $X$-ray transform of a function $f$ also by $Df$  which is defined as 
\begin{align}\label{eq:definition of X-ray transform}
Df(x,\xi) = \int_{\mathbb{R}} f(x+t\xi)d t.
\end{align}
\end{definition}
\begin{definition}[\cite{Sharafutdinov_Book}] \label{Moments_Def}
	The $1^{\mathrm{st}}$ integral moment transform $I^1f$ of a vector field $f \in \mathcal{D}(\mathbb{R}^n;\mathbb{R}^n)$ is defined as a function on $T \mathbb{S}^{n-1}$ by
	\begin{align}\label{eq:definition of integral moment}
	If(x,\xi) = \int_{\mathbb{R}}t  f_{i}(x+t\xi)\xi^{i}d t.
	\end{align}
\par Let  $\g$ be a fixed curve in $\mathbb{R}^n$ parametrized by $\g(s)$ for $ s \in [a, b] \subset \mathbb{R}$ and  let $ \mathcal{C}_{\g}$ denotes the $n$-dimensional subspace of $T\mathbb{S}^{n-1}$ which is just collection all lines intersecting this curve $\g$. Then $\mathcal{C}_{\g}$ can be parametrized by $(\g(s) , \xi)$ for $ t \in [0,1]$ and $ \xi \in \mathbb{S}^{n-1}$. The restrictions of the Doppler transform and of the first integral moment on $\mathcal{C}_{\g}$ is denoted by $D_{\g}$ (restricted Doppler transform)  and by $I_{\g}$ (restricted first integral moment) respectively.
\end{definition}
\par By duality, we can extend the definition of the Doppler transform $D$ as a map between $\mathcal{E}^\prime(\mathbb{R}^n;\mathbb{R}^n)$ (the space of vector fields whose components are compactly supported distributions) and $\mathcal{D}^\prime(T \mathbb{S}^{n-1})$ (the space of vector fields whose components are distributions) in the following way:
\begin{align}\label{eq: definition of ray transform for distributions}
\langle Df , \phi \rangle_{T \mathbb{S}^{n-1}} = \langle f, D^* \phi \rangle_{T^*\Rb^n}, \mbox{ for } f \in \mathcal{E}^\prime(\mathbb{R}^n;\mathbb{R}^n) \mbox{ and } \phi \in \mathcal{D}(T \mathbb{S}^{n-1}),
\end{align}
where $D^*$ is a map from $\mathcal{D}(T \mathbb{S}^{n-1})$ to $\mathcal{E}(\Rb^n;\Rb^n)$ which is defined by
$$(D^* \phi)_i(x) = \int_{\mathbb{S}^{n-1}} \xi^i\left( \int_{-\infty}^\infty \phi(x-t\xi,\xi)dt\right)d\xi.$$
\par In fact, we can also compute the dual of $D_{\g}$  using the duality in the following way:
\begin{align*}
\langle D_\gamma f , \phi \rangle_{\mathcal{C}_{\g}} &= \int_{\mathbb{S}^{n-1}} \int_0^1 D f( \g(s), \xi) \phi (\g(s), \xi) ds d\xi,\\
&=  \int_{\mathbb{S}^{n-1}} \int_0^1 \left(\int_\mathbb{R} f_i( \g(s)+t \xi)\xi^i dt\right) \phi (\g(s), \xi) ds d\xi,\\
&=  \int_{\mathbb{R}^{n}}f_i(x) \left[\int_0^1\frac{(x-\g(s))^i}{|x-\g(s)|^n} \phi \left(\g(s), \frac{x-\g(s)}{|x-\g(s)|}\right) ds\right] dx,\\
&= \langle f, D_{\g}^* \phi\rangle_{\mathbb{R}^n}.
\end{align*}
where $$ (D_{\g}^* \phi)_i(x) = \int_0^1\frac{(x-\g(s))^i}{|x-\g(s)|^n} \phi \left(\g(s), \frac{x-\g(s)}{|x-\g(s)|}\right) ds.$$ 
Also we would like to mention here that in the calculation of $D_{\g}^*$ we assumed that the curve lies outside the support of $f$. We use this dual to extend the definition of restricted Doppler transform for  vector fields whose components are compactly supported distributions.
Now we are going to put some conditions on the curve to make our inversion process work. 
%We say that a source trajectory
% ⊂ (R3\B3 )satisfies Tuy’s condition of order 3, if any plane that passes through B3 intersects
% in at least three points and those points are not located on a line. This means that for any
%p ∈ [−1, 1] and η ∈ S2 there exist at least three values si = si(p, η) ∈ , i = 1, 2, 3, such
%that
%y1 · η = y2 · η = y3 · η = p, yi := y(si(p, η)), i = 1, 2, 3, (2.1)
%and y1 − y3 and y2 − y3 are not collinear
 \begin{definition}[Kirillov-Tuy condition \cite{Denisiuk_2018}]
	Fix a domain $B \subset \Rb^n$. We say that a smooth curve $\g$ satisfies Kirillov-Tuy condition of order $(n-1)$, if for almost every hyperplane $H(\omega, p) = \{x \in \mathbb{R}^n|\langle \omega,x\rangle = p\}$ intersecting the domain $B$, there is a set of points $\g_1,\dots ,\g_{n-1} \in H(\omega, p)\cap \g$, which locally smoothly depends on $(\omega,p)$, such that for almost every point $x \in H(\omega, p)\cap B$ the vectors $x-\g_1, \dots , x-\g_{n-1}$ are  linearly independent. 
\end{definition}
%\begin{definition}[Kirillov-Tuy condition \cite{Denisjuk_Paper}]
%	Fix a domain $B \subset \Rb^n$. We say that a smooth curve $\g$ satisfies Kirillov-Tuy condition of order $ 1$, if for almost every hyperplane $H(\omega, p) = \{x \in \mathbb{R}^n|\langle \omega,x\rangle = p\}$ intersecting the domain $B$, there is a set of points $\g_1,\dots ,\g_{n-1} \in H(\omega, p)\cap \g$, which locally smoothly depends on $(\omega,p)$, such that for almost every point $x \in H(\omega, p)\cap B$ the vectors $x-\g_1, \dots , x-\g_{n-1}$ are  linearly independent. 
%\end{definition}
\begin{example}\cite{Vertgeim2000}
	Consider $\O = B(0, r)\subset \Rb^3$ and $\g$ is a union of 3 great circles on the sphere $S(0, R)$ with $R > \sqrt{3}r$. Then every hyperplane, passing through $\O$, contains in the intersection with $\g$ a nondegenerate triangle.
\end{example}
\begin{example}
	For $(n-1)$-linear independent vectors $v_1, \dots, v_{n-1}$ in $\Rb^n$, define a curve 
	$$ \g = \cup_{i=1}^{n-1}l_i$$
	where $l_i = \{tv_i:t\in \mathbb{R}\setminus 0\}$. Then almost every hyperplane in $\mathbb{R}^n$ will intersect $\g$ at $(n-1)$-distinct points and for almost every $x$ in that hyperplane,  $\{(x-\g_i) | \text{ for }1 \leq i \leq n-1\}$ forms a linearly independent set because the determinant of the matrix $\left((x-\g_i)\right)_{i=1}^{n-1}$ is a multinomial in the variable $x$ whose zero set has measure $0$ in $\mathbb{R}^n$. Hence this $\gamma$ satisfies the Kirillov-Tuy condition of order $(n-1)$.
\end{example}
For the convenience, we will use the following terminology also used by Vertgeim \cite{Vertgeim2000}: 
\begin{definition}[\cite{Vertgeim2000}]
	A curve $\g$ is said to be encompasses a bounded domain $\O$, if $\g \cap \O = \emptyset$ and for each $a \in \g$ and $ \xi \in \Rb^n\setminus \{0\}$ only one of the rays $\{a + t\xi: t\geq 0\}$ and $\{a + t\xi:t\leq 0\}$ can intersect $\O$.
\end{definition}
\begin{remark}
	If a curve $\g$ encompasses the support of a vector field $f$, then the Doppler transform restricted to this curve $\g$ becomes 
	$$D_{\g} f(a, \xi) = \int_0^\infty  f_i(a +t\xi) \xi^i dt.$$
\end{remark}
Our aim is to recover the full vector field when its Doppler transform and first integral moment transform are known for all lines intersecting a fixed curve $\g$ satisfying the Kirillov-Tuy condition of order $1$. To do this first we will recover the Saint-Venant operator of the vector field $f$ which will suffices to recover the solenoidal part of $f$ in $\Rb^n$ because of \cite[Theorem 2.5]{Denisjuk_Paper}.  We then use the solenoidal part to show that the potential part of $f$ can be recovered from the knowledge of the restricted first integral moment transform.

\par The following two decompositions of a vector field are well known from \cite{Sharafutdinov_Book}, the first decomposition is in the full space and the other one is in the bounded domains. These decompositions are valid for any order tensor fields but we state them here for vector fields only.
\begin{theorem}[Theorem 2.14.1 \cite{Sharafutdinov_Book}]\label{th:decomposition theorem in Euclidean space}
	Let $n \geq 2$. For every vector field $f \in \mathcal{E}^\prime(\mathbb{R}^n;\mathbb{R}^n)$ there exist uniquely determined fields $f_{\Rb^n}^s\in \mathcal{S}^\prime(\Rb^n;\mathbb{R}^n)$ and $ v_{\Rb^n} \in \mathcal{S}^\prime(\mathbb{R}^n)$ tending to zero at infinity such that 
	$$ f = f_{\Rb^n}^s + \D v_{\Rb^n}, \quad \quad \delta f^s_{\Rb^n} = 0, $$ 
	The fields $f^s_{\Rb^n}$ and $v_{\Rb^n}$ are smooth outside supp$(f)$ and satisfying the following estimates outside the supp$(f)$:
	\begin{align}\label{eq:estimates in decomposition}
	|f_{\Rb^n}^s(x)|\leq C(1+|x|)^{1-n}, \  |v_{\Rb^n}(x)| \leq C(1+|x|)^{2-n}, \  |\D v_{\Rb^n}(x)| \leq C(1+|x|)^{1-n}.
	\end{align}
	The fields $f_{\Rb^n}^s$ and $ v_{\Rb^n}$ are known as the solenoidal part and the potential part of $f$ respectively.
 \end{theorem}

\begin{theorem}[Theorem 3.3.2, \cite{Sharafutdinov_Book}]\label{th: decomposition of tensor field on manifolds}
	Let $M$ be a compact Riemannian manifold with boundary. Let $k\geq1$ be an integer. Then for any $f\in H^k(M;\mathbb{R}^n)$ there exist uniquely determined $f_M^s\in H^k(M;\mathbb{R}^n)$ and $v_M\in H^{k+1}(M;\mathbb{R}^n)$ such that
	\begin{align}
	f = f_M^s + \D v_M \quad  \text{ with } \quad \delta f_M^s =0 \quad \text{ and } \quad v_M|_{\partial M}=0.
	\end{align}
	The estimates $$\|v_M\|_{k+1}\leq C\| \delta f \|_k \quad \mbox{ and } \|f_M^s\|_k\leq C\|f\|_k$$
	are valid where $C$ is a constant independent of $f$. In particular, $f_M^s$ and $v_M$ are smooth if $f$ is smooth.
\end{theorem}
\begin{definition}[\cite{Sharafutdinov_Book}]
	The Saint-Venant operator $W : \mathcal{E}(\mathbb{R}^n;\mathbb{R}^n)\rightarrow \mathcal{E}(T^*\mathbb{R}^n \otimes T^*\mathbb{R}^n)$ is defined by the equality
	\begin{align}
	(Wf)_{ij}= \frac{1}{2}\left(\frac{\partial f_i}{\partial x^j}- \frac{\partial f_j}{\partial x^i}\right)
	\end{align}
	where $\mathcal{E}(T^*\mathbb{R}^n \otimes T^*\mathbb{R}^n)$ is the space of 2-tensor fields in $\Rb^n$ whose components are smooth functions.
\end{definition}
%\begin{remark}\label{remark:Approximation of distribution}
%We can extend the definition of the Saint-Venant operator for distribution classes also. The  $\mathcal{H}^s(T^*\mathbb{R}^n)$ denotes the Hilbert space defined by Sharafutdinov \cite[Chapter 2]{Sharafutdinov_Book} as completion of Schwartz space $\mathcal{S}(T^*\mathbb{R}^n)$ with respect to norm $||f||_s = ||(1+|y|)^s\hat{f}(y)||_{L^2(T^*\mathbb{R}^n)}$ and $\mathcal{K}^s_K(T^*\mathbb{R}^n)$ denotes the subspace of $\mathcal{H}^s(T^*\mathbb{R}^n)$ having support inside $K$ subset of $\mathbb{R}^n$. Then the Saint-Venant operator $W:\mathcal{H}^s(T^*\mathbb{R}^n)\rightarrow \mathcal{H}^{s-1}(T^*\mathbb{R}^n \otimes T^*\mathbb{R}^n)$ is continuous. 
%\end{remark}
We will need these basic known facts of the Radon transform \cite{Helgason_Book}. For a function $f(x),\  x\in \mathbb{R}^n$, recall the Radon transform of $f$ is given by
\begin{equation*}
f^{\wedge} (\omega,p) = \int_{H(\omega, p)}f(x)ds
\end{equation*}
where $ds$ is the standard volume element on hyperplane $H(\omega, p)$. The dual of Radon transform is defined by
\begin{align*}
f^{\vee}(x) = \int_{x \in P}f(P)d\mu
\end{align*}
where $d \mu$ is the unique probability measure on the set of hyperplanes $P$ containing $x$ invariant under rotations about the point $x$. 
%
%
%\begin{align*}
%f^{\vee}(x) = \int_{\Sb^{n-1}}f(\omega,x \cdot \omega)d\o
%\end{align*}
%where $d\o$ is the usual measure on $\Sb^{n-1}$.
\par By $\mathcal{S}(\mathbb{R}^n)$ we denote the space of smooth function  on $\mathbb{R}^n$ that are rapidly decreasing with all its derivatives. A function $f\in \mathcal{S}(\mathbb{R}^n)$ can be explicitly reconstructed from the knowledge of its Radon transform with the following formula:
\begin{theorem}[\cite{Helgason_Book}]\label{theorem:radon}
	For $f\in \mathcal{S}(\mathbb{R}^n)$, we have
	$$cf(x) = \left(\Lambda \frac{d^{n-1}}{d p^{n-1}} f^{\wedge}(\omega,p)\right)^\vee(x)$$
	where $$ c= (-4\pi)^{\frac{n-1}{2}}\Gamma(n/2)/\Gamma(1/2)$$
	and $$ (\Lambda f)(\omega,p)= \left \{
	\begin{array}{ll}
	f(\omega,p)  &  n \mbox{ odd, }\\
	\mathcal{H}_pf(\omega,p) &  n \mbox{ even  }
	\end{array}                                         \right. $$
	$\mathcal{H}_p$ denotes the Hilbert transform with respect to variable $p$.
\end{theorem}

Recall the following property of the Radon transform:
\begin{equation}\label{eq:property of radon transform}
\left(a_i\frac{\partial}{\partial x^i}f(x)\right)^{\wedge}(\omega,p) = \langle \omega, a \rangle \frac{\partial}{\partial p} f^{\wedge}(\omega, p).
\end{equation}

We now state the main theorem of this article 
\begin{theorem}\label{th:2}
	Let $B$ be a bounded domain in $\Rb^n$  and let $f \in \mathcal{E}^\prime(\mathbb{R}^n;\mathbb{R}^n)$ be a vector field in $\mathbb{R}^n$  supported in $B$. Assume that a curve $\g \subset \mathbb{R}^n$ encompasses $B$ and satisfies the  Kirillov-Tuy condition of order $(n-1)$. Suppose the restricted Doppler transform of $f$, $D_{\g}f$, is known along all lines intersecting the curve $\g$, then the solenoidal part $f^s_{\Rb^n}$  can be determined uniquely.
	\end{theorem}
Next we use the other decomposition (i.e. the decomposition in the bounded domain) of the unknown vector field to construct an elliptic boundary value problem. Using the solution of this constructed elliptic boundary value problem we construct the solenoidal part of $f$ in the bounded domain which we use with the integral moment transform data to get the potential part of $f$.
%With the help of this result, we obtain the full recovery of the vector field $f$ from the first $2$ integral moments, $If$ and $I^{1}f$. 
\begin{theorem}\label{Moments_thm}
Assume $f$, $B$ and $\g$ are same as in Theorem \ref{th:2}. Then the potential part of $f$ can be determined uniquely if its restricted first integral moment transform $I_{\g}f$ and $f^s_{\Rb^n}$ are known.
\end{theorem}
	\begin{remark}
Although, we stated our theorems for any $n \geq 2$ but these theorems can be reformulated for $n=2$ case as  a formally determined problems. The arguments with full data are relatively simpler for the case $n=2$  and have been already studied in \cite{Venky_Suman_Manna}. Therefore, we present proofs only for the case $n \geq 3$.
	\end{remark}
\section{Proofs of Main theorems}\label{Proof of Solenoidal part}
First we are going to prove the Theorem \ref{th:2} for the smooth case following closely the ideas of Denisjuk \cite{Denisjuk_Paper}. Then we use a density argument to prove this theorem for the distribution case. We first start with a Lagrange type interpolation formula in any dimension. 
\begin{proposition}\label{th:1}
	Let the vectors $v_{1},\dots, v_{n-1}\in \Rb^{n-1}$ be linearly independent.  Let the real numbers $y_{1},\dots, y_{n-1}$ be given. Then there exists a unique linear homogeneous polynomial $F(x)=F(x_{1},\dots, x_{n-1})$ such that $F(v_{i})=y_{i}$ for $1\leq i\leq n-1$.
\end{proposition}
\begin{proof}
	We first prove uniqueness before constructing such polynomial. Let us assume that there are two such linear polynomials $F$ and $G$. Then we have $(F-G)(v_i) = 0$ for $i = 1,\dots, (n-1)$. But independence of $v_i$'s implies $F=G$.
	
	Next we construct the polynomial of the following form 
	\[
	F(x)=\sum\limits_{i=1}^{n-1} y_{i} l_{i}(x)
	\]
	with each $l_{i}$ is a homogeneous linear polynomial in $x$ and defined by
	\[
	l_{i}(x)= \frac{\Delta_{i}(x)}{\Delta}
	\]
	where
	\[
	\Delta_{i}(x)=\det 
	\begin{pmatrix}
	&  v_{11} & \dots & v_{1 n-1}\\
	& \vdots & \ & \vdots \\
	& v_{i-1 1} & \dots & v_{i-1 n-1}\\
	&  x_{1} & \dots &  x_{n-1}\\
	& v_{i+1 1} & \dots & v_{i+1 n-1}\\
	& \vdots& \ & \vdots \\
	& v_{n-11} &\dots & v_{n-1n-1}
	\end{pmatrix}
	\]
	and 
	\[
	\Delta= \det 
	\begin{pmatrix}
	&  v_{11} & \dots & v_{1 n-1}\\
	& \vdots& \ & \vdots\\
	&  v_{n-11} & \dots & v_{n-1n-1}
	\end{pmatrix}.
	\]

	Then, we have 
	\[
	F(x)=\sum\limits_{i=1}^{n-1} y_{i} l_{i}(x)
	\]  is a homogeneous linear polynomial in $x$ satisfying $F(v_{i})=y_{i}$ for each $1\leq i \leq n-1$.
\end{proof}
\par If  $\xi$ and $\eta$ are two arbitrary vectors in $ \mathbb{R}^n$, then the action of $Wf$ on any such pair is defined by $ \langle Wf , (\xi,\eta )\rangle =  \langle Wf , \xi\otimes \eta \rangle = (Wf)_{ij}\xi^{i}\eta^{j}$. As in \cite{Denisjuk_Paper}, we use the Radon inversion formula given in Theorem \ref{theorem:radon} to get $Wf$. With this in mind, we compute $\frac{\partial}{\partial p} \langle Wf , (\xi,\eta)\rangle ^{\wedge}(\omega,p)$ for an arbitrary pair of vectors  $\xi, \eta \in \mathbb{R}^n$.
\begin{lemma}[\cite{Denisjuk_Paper}]\label{lemma:2.1}
	Let the vectors $\xi, \eta \in \mathbb{R}^n$ be given. Then for any vector field $f \in\mathcal{D}(\mathbb{R}^n;\mathbb{R}^n)$
	\begin{equation}\label{eq:Saint venant}
	\langle  Wf , (\xi,\eta )\rangle ^{\wedge}(\omega,p)
	=\frac{1}{2}\left(\langle \omega, \eta\rangle \frac{\partial }{\partial p}\langle f, \xi\rangle^{\wedge}(\omega,p)-\langle \omega, \xi\rangle \frac{\partial }{\partial p}\langle f, \eta\rangle^{\wedge}(\omega,p)\right).
	\end{equation}
\end{lemma}
\begin{remark}\label{remark:2.2}
	If we decompose $\xi$ and $\eta$ as
	$$\xi =\xi_1+\xi_2 \quad \mbox{ and } \quad \eta =\eta_1+\eta_2$$
	where $\xi_1$, $\eta_1$ are parallel to $\omega$ and $\xi_2$, $\eta_2$ are orthogonal to $\omega$.  From Lemma \ref{lemma:2.1}, we can see that the L.H.S. of equation \eqref{eq:Saint venant} vanishes if both vectors $\xi$, $\eta$ are orthogonal to $\omega$ or parallel to $\omega$. Therefore we can see that 
	$$\langle Wf, (\xi, \eta)\rangle = \langle Wf, (\xi_1, \eta_2)\rangle +\langle Wf, (\xi_2, \eta_1)\rangle.$$ 
	Keeping this in mind it is sufficient to consider the following case:
\end{remark}
\begin{corollary}[\cite{Denisjuk_Paper}]\label{corollary:2.1}
	Let vector $\xi\in \mathbb{R}^n$ be orthogonal to $\omega$. Then for any vector field
	$f \in \mathcal{D}(\mathbb{R}^n;\mathbb{R}^n)$
	\begin{align}\label{eq:12}
	\langle Wf , (\xi, \omega)\rangle ^{\wedge}(\omega,p) =2^{-1} \frac{\partial}{\partial p}\langle f, \xi\rangle^{\wedge}(\omega, p).
	\end{align}
\end{corollary}
Using the above Corollary \ref{corollary:2.1} to prove Theorem \ref{th:2} (atleast in the smooth case) it is sufficient to show that the derivatives of $\frac{\partial }{\partial p}\langle f, \xi\rangle^{\wedge}(\omega,p)$ w.r.t. $p$ can be explicitly calculated in terms of the restricted Doppler transform for an arbitrary vector $\xi$ parallel to $H(\omega, p)$.
\begin{lemma}\label{lemma:2.2}
	Let $\g_0$ be an intersection point of the curve and fixed hyperplane $H(w,p)$.
	Consider the parametrization in $\mathbb{R}^n$: $x= \g_{0}+\xi t$, where $\xi \in \mathbb{S}^{n-1}$ and $t \in [0, \infty)$.  Then  for any vector field $f\in \mathcal{D}(\mathbb{R}^n;\mathbb{R}^n)$ and a weight $w(\xi)$
	\begin{align*}
	\frac{\partial^{n-1}}{\partial p^{n-1}} [f_j(\g_0+t\xi)\xi^j t^kw(\xi) ]^{\wedge}(\omega,p) \quad \text{ for }k=1, 2
	\end{align*}
	can be computed explicitly in terms of $D_{\g}f$.
\end{lemma}
\begin{proof}
	The proof of this lemma is given in \cite{Denisjuk_Paper} for any tensor fields of any rank and the $n=3$ case. We give it here for the sake of completeness in our case.
	
	\begin{align*} 
	D_{\g}f(\g_0,\xi) &= \int_0^\infty f_j(\g_0+t\xi)\xi^j dt\\
	D_{\g}f(\g_0,\xi) &= \int_0^\infty\frac{f_j(\g_0+t\xi)\xi^j t^{n-2}}{t^{n-2}} dt.
	\end{align*}
	Now integrating the above equation over the unit sphere centered at $\g_0$ and perpendicular to $\o$, $\mathbb{S}(\omega) = \{\xi \in \mathbb{S}^{n-1}: |\g_{0}-\xi| =1 \text{ and } \langle \xi, \omega\rangle = 0\}$, where $|\g_0 -\xi|$ is the Euclidean distance between $\g_0$ and $\xi$.
	\begin{align*}
	\int_{\mathbb{S}(\omega)} D_{\g}f(\g_0,\xi) d\omega(\xi)&= \int_{\mathbb{S}(\omega)}\int_0^\infty
	\frac{f_j(\g_0+t\xi)\xi^j t^{n-2}}{t^{n-2}} dt d\omega(\xi)\\
	&= \int_{H_{(\omega,p)}}\frac{f_j(\g_0+t\xi)\xi^j }{t^{n-2}} ds,
	\quad \mbox{ where } ds = t^{n-2} dt d\omega(\xi)\\
	&= [f_j(\g_0+t\xi)\xi^j t^{-n+2}]^{\wedge}(\omega,p).
	\end{align*}
	where $$[f_j(\g_0+t\xi)\xi^j t^{-n+2}]^{\wedge}(\omega,p) = \int_{H(\omega,p)} f_j(y)(y-\gamma_0)^j\frac{1}{|y-\gamma_0|^{n-1}}dy.$$
	Let us multiply $D_{\g}f(\g_0,\xi)$ by the weight function $w(\xi)$ and apply the differential operator $L = \omega_i\frac{\partial}{\partial \xi_i} =  t\omega_i\frac{\partial}{\partial x_i}$(Einstein's summation is assumed) to the product $w(\xi)D_{\g}f(\g_0,\xi)$. 
	\begin{align*}
	\int_{\mathbb{S}(\omega)}L\left(w(\xi) D_{\g}f(\g_0,\xi)\right) d\omega(\xi)
	&= \int_{\mathbb{S}(\omega)}\omega_i\frac{\partial}{\partial \xi_i}(w(\xi) D_{\g}f(\g_0,\xi)) d\omega(\xi)\\
	&= \int_{\mathbb{S}(\omega)} \omega_i\frac{\partial}{\partial \xi_i}\left(w(\xi) \int_0^\infty f_j(\g_0+t\xi)\xi^j dt\right) d\omega(\xi)\\
	&= \int_{\mathbb{S}(\omega)} \int_0^\infty \omega_i\frac{\partial}{\partial \xi_i}\left(w(\xi)  f_j(\g_0+t\xi)\xi^j\right) dt d\omega(\xi)\\
	&= \int_{\mathbb{S}(\omega)} \int_0^\infty t\omega_i\frac{\partial}{\partial x_i}\left(w(\xi)  f_j(\g_0+t\xi)\xi^j\right) dt d\omega(\xi)\\
	&= \int_{H(\omega, p)}\omega_i\frac{\partial}{\partial x_i}(w(\xi)  f_j(\g_0+t\xi)\xi^j)t^{3-n} ds\\
	&= \frac{\partial}{\partial p}\int_{H(\omega, p)}(w(\xi)  f_j(\g_0+t\xi)\xi^j)t^{3-n} ds \text{, using(\ref{eq:property of radon transform}) }\\
	&= \frac{\partial}{\partial p} [f_j(\g_0+t\xi)\xi^j w(\xi)t^{3-n}]^{\wedge}(\omega,p).
	\end{align*}
	Application of the same differential operator $L$ successively $n-2$ and $n-3$ more times will give us the following two expressions respectively:
	\begin{align}\label{eq:14}
	\int_{\mathbb{S}(\omega)}L^{n-1}(w(\xi) D_{\g}f) d\omega(\xi)= \frac{\partial^{n-1}}{\partial p^{n-1}} [f_j(\g_0+t\xi)\xi^j w(\xi) t]^{\wedge}(\omega,p).
	\end{align}
	\begin{align}\label{eq:f}
	\int_{\mathbb{S}(\omega)}L^{n-2}(w(\xi) D_{\g}f) d\omega(\xi)= \frac{\partial^{n-2}}{\partial p^{n-2}} [f_j(\g_0+t\xi)\xi^j w(\xi) ]^{\wedge}(\omega,p).
	\end{align}
	Let us differentiate the above equation \eqref{eq:f} with respect to $p$ by considering planes parallel to $H(\omega, p)$:
	\begin{align*}
	\frac{\partial}{\partial p}\int_{\mathbb{S}(\omega)}L^{n-2}(w(\xi) D_{\g}f) d\omega(\xi)&= \frac{\partial^{n-2}}{\partial p^{n-2}} \left[\frac{\partial \lambda}{\partial  p}\frac{\partial }{\partial \lambda}f_j(\g(\lambda)+t\xi)\xi^j w(\xi) \right]^{\wedge}(\omega,p)\\
	&\quad \quad +\frac{\partial^{n-1}}{\partial p^{n-1}} [f_j(\g_0+t\xi)\xi^j w(\xi) ]^{\wedge}(\omega,p).
	\end{align*}
	Now if we define $\tilde{L}= \frac{\partial \lambda}{\partial  p}L^{n-2}\frac{\partial }{\partial \lambda}$ then using similar analysis, we get
	\begin{align*}
	\frac{\partial^{n-2}}{\partial p^{n-2}} \left[\frac{\partial \lambda}{\partial  p}\frac{\partial }{\partial \lambda}f_j(\g(\lambda)+t\xi)\xi^j w(\xi) \right]^{\wedge}(\omega,p) = \int_{\mathbb{S}(\omega)}\tilde{L}(w(\xi) D_{\g}f) d\omega(\xi).
	\end{align*}
	Therefore 
	\begin{align*}
	\frac{\partial^{n-1}}{\partial p^{n-1}} [f_j(\g_0+t\xi)\xi^j w(\xi) ]^{\wedge}(\omega,p) &=\frac{\partial }{\partial p}\int_{\mathbb{S}(\omega)}L^{n-2}(w(\xi) D_{\g}f) d\omega(\xi)\\
	&\quad \quad-\int_{\mathbb{S}(\omega)}\tilde{L}(w(\xi) D_{\g}f) d\omega(\xi).
	\end{align*}
\end{proof}
\begin{lemma}\label{lemma:2.3}
	Let $\g$ be a curve satisfying the Kirillov-Tuy condition of order $(n-1)$ and $D_{\g}f$ is known for all lines intersecting $\g$. Then for any vector field $f \in \mathcal{D}(\mathbb{R}^n;\mathbb{R}^n)$ and any vector $\textit{\textbf{v}}$ parallel to $H(\omega, p)$, we have
	\begin{align}\label{eq:15}
	\frac{\partial^{n-1}}{\partial p^{n-1}}\langle f(x), \textit{\textbf{v}}\rangle^\wedge(\omega,p)
	\end{align}
	can be computed explicitly in terms of $D_{\g}f$.
\end{lemma}
\begin{proof}
	From our Kirillov-Tuy condition	we know that for almost every hyperplane  $H(\o,p)$ there exist points  $\g_{1},\cdots,\g_{n-1}$ on the curve $\g$ and $H(\o,p)$ satisfying generic condition. Let us fix one such hyperplane  $H(\o,p)$. Then from our assumption the following integrals are known:
	$$ \int_0^\infty  f_j(\g_i+t_i\xi) \xi^j dt_i \text{ , for } i= 1 ,\dots , n-1\mbox{ and } \forall \xi \in \mathbb{S}^{n-1}.$$ 
	For $x\in H(\o,p)$, we can write $x=\g_{j}+t_{j}\xi_{j}$ for $1\leq j\leq n-1$ for some $t_{j}\in \Rb^{+}$ and $\xi_{j}\in \Sb^{n-1}$. For any point $x\in H(\o,p)$ and for a vector $\vi$ parallel to  $H(\o,p)$, using Proposition \ref{th:1}, we have 
	\begin{align*}
	\langle f(x),\vi\rangle&= \sum\limits_{i=1}^{n-1} \langle f(x),(x-\g_{i})\rangle \frac{\Delta_{i}(\vi)}{\Delta}
	\end{align*}
	\begin{align}\label{eq:ab}
	= \sum_{i=1}^{n-1} \langle f(x), \xi_i\rangle t_i \frac{\Delta_i(\textit{\textbf{v}})}{\Delta},
	\end{align}
	where $ x-\g_i =\  \xi_it_i$ for every intersection point $\g_i$, $i =1, \dots ,n-1$ and $ \xi_jt_j = \g_{ij} +  \xi_it_i$ with $\g_{ij} = \g_i -\g_j$. Recall  \\
	\begin{align*}
	\Delta_i(\textit{\textbf{v}}) &= \det \begin{pmatrix}
	\g_{i1}+\xi_it_i\\
	\vdots\\
	\textit{\textbf{v}}\\
	\vdots\\
	\g_{i(n-1)}+\xi_it_i
	\end{pmatrix}\leftarrow i^{th} \text{ row}\quad \text{ and }\\\\
	\Delta &= \det \begin{pmatrix}
	\g_{i1}+\xi_it_i\\
	\vdots\\
	\xi_it_i\\
	\vdots\\
	\g_{i(n-1)}+\xi_it_i
	\end{pmatrix}=\det \begin{pmatrix}
	\g_{i1}\\
	\vdots\\
	\xi_it_i\\
	\vdots\\
	\g_{i(n-1)}
	\end{pmatrix}.
	\end{align*}
	After simplifying this, we can see that $ \frac{\Delta_i(\textit{\textbf{v}})}{\Delta} = w(\xi_i)+t_i\tilde{w}(\xi_i) $ where $w(\xi_i)$ and $\tilde{w}(\xi_i)$ are function of $\xi_i$ and known quantities $\g_{ij}$. Putting this back to equation \eqref{eq:ab}, we get 
	\begin{align*}
	\langle f(x),\vi\rangle = \sum_{i=1}^{n-1} \langle f(x), \xi_i\rangle(w(\xi_i)+t_i\tilde{w}(\xi_i)).
	\end{align*}
	Then we have the following:
	\begin{align*}
	\frac{\partial^{n-1}}{\partial p^{n-1}}\langle f(x), \textit{\textbf{v}}\rangle^\wedge(\omega,p) &= \sum_{i=1}^{n-1}\frac{\partial^{n-1}}{\partial p^{n-1}} \int_{H(\omega,p)} \langle f(x), \xi_i\rangle(w(\xi_i)+t_i\tilde{w}(\xi_i))ds\\
	&= \sum_{i=1}^{n-1}\left\{\frac{\partial}{\partial p}\int_{\mathbb{S}(\omega)}L^{n-2}(w(\xi_i) D_{\g}f) d\omega(\xi_i)\right.\\
	&\quad \quad \quad -\int_{\mathbb{S}(\omega)}\tilde{L}(w(\xi_i) D_{\g}f) d\omega(\xi_i)\\
	&\quad \quad \quad +\left.\int_{\mathbb{S}(\omega)}L^{n-1}(\tilde{w}(\xi_i) D_{\g}f) d\omega(\xi_i)\right\}.
	\end{align*}
\end{proof}
\subsection{Proof of Theorem \ref{th:2}}
\begin{proof}[Proof of Theorem \ref{th:2}]
	We fix a hyperplane $H(\omega, p)$ and consider arbitrary vectors $\xi$, $\eta$. Decompose these vectors in the components parallel to $\omega$ and  orthogonal to $\omega$ as 
	\begin{align*}
	\xi =\xi_{1}+ \xi_{2}, \quad  \quad \eta= \eta_{1} + \eta_{2}, 
	\end{align*}
	where $\xi_{1}$, $\eta_{1}$ are orthogonal to $\omega$, while $\xi_{2}, \eta_{2}$ are parallel to $\omega$.
	Using Lemma \ref{lemma:2.1} and Corollary \ref{corollary:2.1} implies that,
	\begin{align*}
	\langle  Wf , (\xi,\eta )\rangle ^{\wedge}(\omega,p)&= \left[\langle  Wf , (\xi_1,\eta_1 )\rangle +\langle  Wf , (\xi_1,\eta_2 )\rangle+\langle  Wf , (\xi_2,\eta_1 )\rangle\right.\\
	&\quad \quad \quad +\left.\langle  Wf , (\xi_2,\eta_2 )\rangle\right]^{\wedge}(\omega,p)
	\end{align*}
	From Remark \ref{remark:2.2}, it is known that the first and last term in the above expression are zero. Therefore, we have
	\begin{align*}
	\langle  Wf , (\xi,\eta )\rangle ^{\wedge}(\omega,p)&= \left[\langle  Wf , (\xi_1,\eta_2 )\rangle+\langle  Wf , (\xi_2,\eta_1 )\rangle\right]^{\wedge}(\omega,p)\\
	&=\left[\langle  Wf , (\xi_1,\eta_2 )\rangle-\langle  Wf , (\eta_1,\xi_2 )\rangle\right]^{\wedge}(\omega,p)\\
	&=  \langle\eta , \omega\rangle\left[\langle  Wf , (\xi_1,\omega )\rangle\right]^{\wedge}(\omega,p)- \langle\xi , \omega\rangle\left[\langle  Wf , (\eta_1,\omega)\rangle\right]^{\wedge}(\omega,p).
	\end{align*}
	Now from Corollary \ref{corollary:2.1}, we have 
	\begin{align*}
	\frac{\partial^{n-1}}{\partial p^{n-1}}\langle  Wf , (\xi,\eta )\rangle ^{\wedge}(\omega,p)
	&=  \frac{1}{2}\left(\langle\eta , \omega\rangle\frac{\partial^{n}}{\partial p^{n}}\langle f , \xi_1 \rangle^{\wedge}(\omega,p)- \langle\xi , \omega\rangle\frac{\partial^{n}}{\partial p^{n}}\langle  f , \eta_1\rangle^{\wedge}(\omega,p)\right).
	\end{align*}
	Finally,  we use Theorem \ref{theorem:radon} to get
	\begin{align}\label{eq:Saint Venant operator in terms of Doppler transform}
	c\langle Wf(x),(\xi,\eta) \rangle 
%	&=
%	\left[\Lambda \frac{\partial^{n-1}}{\partial p^{n-1}}\langle  Wf , (\xi,\eta )\rangle ^{\wedge}(\omega,p)\right]^{\vee}(x),\\
	= \left[ \frac{1}{2}\Lambda \left(\langle\eta , \omega\rangle\frac{\partial^{n}}{\partial p^{n}}\langle f , \xi_1 \rangle^{\wedge}(\omega,p)- \langle\xi , \omega\rangle\frac{\partial^{n}}{\partial p^{n}}\langle  f , \eta_1\rangle^{\wedge}(\omega,p)\right)\right]^{\vee}(x)
	\end{align}
	where $\Lambda$ and $c$ are the same as in Theorem \ref{theorem:radon}. Hence we have recovered $Wf$ because  R.H.S. is known from Lemma \ref{lemma:2.3}.  
	After this we can use the formula from \cite[Theorem 2.5]{Denisjuk_Paper} for $m=1$ to get the solenoidal part $f_{\Rb^n}^s$ explicitly from $Wf$. Therefore combining all these we find an operator $(D_{\g})^{-1}$  such that
\begin{align}\label{eq:Inverse of restricted Doppler transform}
 \mathtt{S} (f) = f^s_{\Rb^n} = D_{\g}^{-1}D_{\g}f,\quad \text{ for } f \in \mathcal{D}(\mathbb{R}^n;\mathbb{R}^n) .
\end{align}
This proves our theorem for the smooth case. 

Next we claim that the same formula is also valid for compactly supported distributions. We extend   $D_{\g}^{-1}D_{\g} : \mathcal{E}^\prime(\mathbb{R}^n;\mathbb{R}^n) \rightarrow  \mathcal{D}^\prime(\mathbb{R}^n;\mathbb{R}^n)$ using duality. Then we want to prove the following:
\begin{claim}
\begin{align}\label{eq:Inverse of restricted Doppler transform for distribution}
 \mathtt{S} (f) = f^s_{\Rb^n} = D_{\g}^{-1}D_{\g}f,\quad \text{ for } f \in \mathcal{E}^\prime(\mathbb{R}^n;\mathbb{R}^n) .
\end{align}
\end{claim}
\textit{Proof of Claim: } For $f \in \mathcal{E}^\prime(\mathbb{R}^n;\mathbb{R}^n)$, we can a find a sequence of $f_k \in \mathcal{D}(\mathbb{R}^n;\mathbb{R}^n)$ which converges to $f$ in the distribution sense. Using the continuity of the Radon transform and the inverse radon transform on distributions \cite[Section 4]{Ludwig1966}, we can see that the right hand side of \eqref{eq:Saint Venant operator in terms of Doppler transform} makes sense for $f \in \mathcal{E}^\prime(\mathbb{R}^n;\mathbb{R}^n)$. We can take limit $k \rightarrow \infty $ on the left hand side \eqref{eq:Saint Venant operator in terms of Doppler transform} also because of \cite[Lemma 2.5.2]{Sharafutdinov_Book}. Hence the \eqref{eq:Saint Venant operator in terms of Doppler transform} is valid for compactly supported distributions also.

Then we can apply theorem \cite[Theorem 2.5]{Denisjuk_Paper} (this result is true  for tempered distributions also by duality) for  $\mathcal{E}^\prime(\mathbb{R}^n;\mathbb{R}^n)$  to conclude that 
$f^s_{\Rb^n} = D_{\g}^{-1}D_{\g}f$, $\text{ for } f \in \mathcal{E}^\prime(\mathbb{R}^n;\mathbb{R}^n)$.  This completes the proof of our theorem.
\end{proof}
\begin{corollary}\label{corollary:function recovery}
	Let $v$ be a compactly supported distribution whose support is contained in $B$ and the ray transform of $v$ is known for all lines intersecting the curve $\g$ then $v$ can be reconstructed with an explicit formula.
\end{corollary}
\begin{proof}
	The proof will follow for the smooth case directly from the arguments in the Lemma \ref{lemma:2.2} for functions instead of vector fields. And then using the limiting argument we can conclude for the general case.
\end{proof}
\subsection{Proof of Theorem \ref{Moments_thm}}
\begin{proof}[Proof of Theorem \ref{Moments_thm}]
 As we know from our assumptions on $f$ that $f$ is supported inside $B$ and we also that form the decomposition theorem \ref{th:decomposition theorem in Euclidean space} that $v_{\Rb^n}$ and $f^s_{\Rb^n}$ are smooth outside $B$. Thus we have 
	 $$ f^s_{\Rb^n}(x) = - \D v_{\Rb^n}(x), \quad \mbox{ for } x \in \Rb^n\setminus B$$
	 and 
	 $$ |f^s_{\Rb^n}(x)|\leq C(1+|x|)^{1-n},\quad |v_{\Rb^n}(x)|\leq C(1+|x|)^{2-n}, \mbox{  for } x \in \Rb^n\setminus B$$
Therefore $\D v_{\Rb^n}(x)$ is known for every $x\in \Rb^n\setminus B$ because $f^s_{\Rb^n}$ is known from $D_{\g}f$. We can  see from the above estimates that the integral of $f^s_{\Rb^n}$ along lines make sense only if $n \geq 3$ (This is why we have to use different arguments for $n=2$ case). Fix any point $x_0$ on $\Rb^n \setminus B$  and a direction vector $\xi \in \Sb^{n-1}$ such that the ray $\{x_0+t \xi: t \in [0, \infty)\}$ lie outside $B$. Consider 
	 \begin{align*}
	 \int_0^\infty (\D v_{\Rb^n}(x_0+t \xi))_i\xi^i dt &= -\int_0^\infty \langle f^s_{\Rb^n}(x_0+t \xi), \xi \rangle dt\\
	 \int_0^\infty \frac{d}{d t}( v_{\Rb^n}(x_0+t \xi) dt &= -\int_0^\infty \langle f^s_{\Rb^n}(x_0+t \xi), \xi \rangle dt\\
	 v_{\Rb^n}(x_0)&= \int_0^\infty \langle f^s_{\Rb^n}(x_0+t \xi), \xi \rangle dt.
	 \end{align*}
	 Since the R.H.S. of the last equality is known hence $v_{\Rb^n}(x_0)$ is known for all $x_0 \in \Rb^n \setminus B$. In particular $v_{\Rb^n}|_{\partial B}$ is known.
	 \par Next we decompose $f$ using Theorem \ref{th: decomposition of tensor field on manifolds} as the following:
	 $$ f = f^s_B +\D v_B, \quad \delta f^s_B = 0, \quad v_B|_{\partial B} = 0.$$
	 On $B$, we have
	 \begin{align}\label{eq:relation between decompositions}
	 f^s_{\Rb^n} - f^s_B = \D (v_B -v_{\Rb^n})
	 \end{align}
	 Applying divergence operator $\delta$ on both side, we get
	 $$\delta \D (v_B -v_{\Rb^n}) = 0$$
	 $\delta \D$ is the Laplace operator. Now we solve the following elliptic boundary value problem to get $v_B -v_{\Rb^n}$:
	 \begin{align*}
	 \delta \D (v_B -v_{\Rb^n}) = 0, \quad \mbox{ in } B\\
	 (v_B -v_{\Rb^n}) = g,  \quad \mbox{ on } \partial B,
	 \end{align*}
	 where $g(x) = -\int_0^\infty \langle f^s_{\Rb^n}(x+t \xi), \xi \rangle dt$ with $\xi$ is an outward normal at $x \in \partial B$. 
	 \par Now $f^s_B $  can be computed from the equation \eqref{eq:relation between decompositions}  once we know  $(v_B -v_{\Rb^n})$ in $B$ from the above elliptic boundary value problem. In the calculation below, we can assume $v_B$ and $f^s_B $ are zero outside $B$ because the support of $f$ is contained in $B$. Now fix a point $\g_0$ on the curve $\g$ and consider the following:
	\begin{align*}
	If (\g_0, \xi)&= If^s_{B}(\g_0, \xi) + I(\D v_{B})(\g_0, \xi)\\
	I(\D v_{B})(\g_0, \xi)&=	If (\g_0, \xi) - If^s_{B}(\g_0, \xi)\\
	\int_0^\infty t \frac{\partial v_{B}}{\partial x^i}(\g_0 +t \xi)\xi^i dt&=	If (\g_0, \xi) - If^s_{B}(\g_0, \xi)\\
		\int_0^\infty t \frac{d }{dt}\left\{v_{B}(\g_0+t \xi)\right\} dt&=	If (\g_0, \xi) - If^s_{B}(\g_0, \xi)\\
			-	\int_0^\infty v_{B}(\g_0 +t \xi) dt&=	If (\g_0, \xi) - If^s_{B}(\g_0, \xi)\\
		\int_0^\infty v_{B}(\g_0 +t \xi) dt&=If^s_{B}(\g_0, \xi)-	If (\g_0, \xi) \\
			Dv_{B}(\g_0 , \xi) &=If^s_{B}(\g_0, \xi)-	If (\g_0, \xi).
	\end{align*}	 	 
Therefore $Dv_{B}(\g_0, \xi)$ is known because R.H.S. is known from above discussion. Which implies that        $Dv_{B}$ is known for all lines intersecting the curve $\g$  since $\g_0$ was an arbitrary point of the curve $\g$ and $\xi \in \Sb^{n-1}$. Hence $v_B$ can be reconstructed from the Corollary \ref{corollary:function recovery} which completes the proof of the Theorem \ref{Moments_thm}.
\end{proof}

\bibliographystyle{plain}
\bibliography{Restricted_Doppler_tomography}
\end{document}

%% file: preamble.tex
\newtheorem{theorem}{Theorem}[]
\newtheorem{lemma}[theorem]{Lemma}
\newtheorem{proposition}[theorem]{Proposition}
\newtheorem{corollary}[theorem]{Corollary}
\newtheorem{definition}[theorem]{Definition}
\newtheorem{remark}[theorem]{Remark}

\newtheorem{claim}[theorem]{Claim}

\newcommand{\cout}[1]{}